\documentclass[11pt,draft, reqno]{amsart}

\usepackage[notcite,notref]{showkeys}

\allowdisplaybreaks

\usepackage[margin=18mm]{geometry}
\usepackage{amssymb, amsmath, amsthm, amscd}

\usepackage[T1]{fontenc}
\usepackage{eucal,mathrsfs,dsfont}%gives nicer set-names. Use \mathds instead of \mathds
\usepackage{color}%cyan,red;magenta,green;yellow,blue
\def\<{\langle}
\def\>{\rangle}

\renewcommand{\leq}{\leqslant}
\renewcommand{\geq}{\geqslant}
\renewcommand{\le}{\leqslant}
\renewcommand{\ge}{\geqslant}

\def\EE{{\mathcal E}}
\def\FF{{\mathcal F}}
\def\LL{{\mathcal L}}

\newcommand{\D}{\mathscr{D}}
\newcommand{\E}{\mathcal{E}}
\newcommand{\F}{\mathscr{F}}
\newcommand{\HH}{\mathcal{H}}

\definecolor{mno}{rgb}{0.5,0.1,0.5}

\newcommand{\R}{\mathds R}

\newcommand{\e}{\varepsilon}

\newcommand{\Pp}{\mathds P}
\newcommand{\Ee}{\mathds E}

\newcommand{\w}{\omega}

\newtheorem{theorem}{Theorem}[section]

\newtheorem{proposition}[theorem]{Proposition}

\theoremstyle{definition}

\newtheorem{remark}[theorem]{Remark}
\newtheorem{assumption}[theorem]{Assumption}

\numberwithin{equation}{section}

\begin{document}
\allowdisplaybreaks
\title[Stochastic homogenization of
stable-like processes with divergence free drifts] {\bfseries Stochastic homogenization of stable-like processes with divergence free drifts}
\author{Xin Chen \quad \hbox{and}\quad Kun Yin}

\date{}

\maketitle

\begin{abstract}
In this paper we will study homogenization of
 stable-like processes with divergence-free drifts in
ergodic environments. In particular, neither the drifts nor the
stream functions are required to be bounded.

\medskip


\medskip

\end{abstract}
\allowdisplaybreaks

\section{Introduction and results}\label{section1}

\subsection{Background}

Homogenization theory aims to provide rigorous macroscopic characterizations of microscopically heterogeneous media. Papanicolaou and Varadhan
\cite{PV}, Kozlov \cite{Koz} independently proved the homogenization for the (parabolic or elliptic) equations corresponding to
second order elliptic operators with divergence form which have random coefficients in ergodic environments for the first time. They introduced
an important method usually called "seen from particle" to construct the correctors in ergodic environments
which satisfy the properties of everywhere sublinearty. This
was the crucial point to establish enhancements of the diffusive coefficients in limit equations.
Later on a lot of homogenization problems  were investigated  for various (linear and non-linear) elliptic and parabolic differential equations as well as system of equations
in ergodic environments. See also \cite{BMW,CD,KLO,ZP} and reference therein about
the homogenization for symmetric processes in ergodic environments whose infinitesimal generators
were second order differential operators with divergence forms.

Osada \cite{O} studied the homogenization of stochastic turbulence flows in ergodic environments, which 
under the assumption that both the drifts and associated stream functions
(which were the function $\tilde H$ in Assumption \ref{a1-2} below) were bounded. The method of corrector was also used
in \cite{O} to show the enhanced variance for limiting Brownian motion.
After \cite{O}, there were various
different results concerning the homogenization of stochastic turbulence flows with unbounded drifts and stream functions.
Oelschl\"ager \cite{Oe} proved
the convergence in Wasserstain metric for associated scaled processes. The quenched invariance principle
was established by Fannjiang and Komorowski \cite{FK} which used maximal inequalities for
associated Poisson equations with only the moments(not the uniform bounds) of stream functions.
The result in \cite{FK} has been extended by Fannjiang and Komorowski \cite{FK1} to the case in time-space ergodic environments.
Fannjiang and Papanicolaou \cite{FP} investigated the convergence of associated Fokker-Planck equations and introduced
some conditions on the correlation functions of drifts for existence of stream functions, see also Komorowski and Olla \cite{KO}
for the case with time-dependent drifts. Fehrman \cite{Fe} also studied the time-dependent model, where
where
mixing properties of time variables may induce different terms in limit processes. We refer the reader to monograph
\cite{KLO} for detailed introduction on this subject.

Recently, there were much important progress on the homogenization for stochastic turbulence flows
whose drifts were critically correlated Gaussian fields. For this case, stream functions may not exist, and
the idea of quantitative homogenization, as well as renormalization technique were applied
to study several problems. Chatzigeorgiou, Morfe, Otto, and Wang \cite{CMOW} obtained
the large time moment estimates for critically correlated stochastic turbulence flows.
Armstrong, Bou-Rabee and Kuusi \cite{ABK} have characterized the quenched invariance principle, where
the scaling orders were not diffusive and limit processes were Brownian motions whose coefficients
only depend on the drifts(did not depend on original driven Brownian motion). Otto and Wagner \cite{OW}
investigated the intermittent behaviors for critically correlated stochastic turbulence flows,
and Morfe, Otto and Wagner \cite{MOW} studied the multi-scale homogenization for this model.

Now we give a brief review of some recent work on
homogenization problems
for non-local operators.
Piatnitski and Zhizhina \cite{PZ}, Flegel, Heida and Slowik \cite{FHS}
studied homogenization problem
for a class of non-local operators in ergodic environments whose jump kernels were $L^2$-integrable, and see
\cite{PZ1} for corresponding result in time-space ergodic environments.
When the jump kernels are not $L^2$-integrable, associated process has heavy tail property and
the scaling orders and limiting processes are no longer diffusive.
Chen, Kim and Kumagai \cite{CKK} proved the Mosco convergence of non-local Dirichlet forms associated with
symmetric stable-like random walks in independent random
conductance model, where limiting processes
were $\alpha$-stable L\'evy processes. Furthermore, Chen, Kumagai and Wang \cite{CKW} established quenched invariance
principle for symmetric stable-like random walks in independent random conductance model.
Based on a new ABP type estimate, Schwab \cite{S} studied
homogenization of a class of fully
non-linear integral-differential equations in ergodic environments.
Kassmann, Piatnitski and Zhizhina \cite{KPZ} investigated
homogenization of a class of symmetric stable-like processes in
 ergodic environments whose jump kernels were uniformly elliptic and
 having product form. See also Flegel and Heida \cite{FH} for corresponding results about
 the discrete counterpart of this model.
 By proving a new criterion for pre-compactness in $L^1$ space,
 Chen, Chen, Kumagai and Wang \cite{CCKW1} established homogenization for symmetric
 stable-like processes with more general forms, whose jump kernels could be degenerate, unbounded,
 and singular, see also \cite{CCKW} for another special example. But by our knowledge, there are not
 any results about homogenization of stochastic turbulence flows in ergodic environments driven by
 jump processes,
 \emph{the main purpose of this paper is to study  stochastic homogenization for stable-like processes with divergence-free drifts in
ergodic environments.}

%We refer the reader to \cite{BGG,BRS,PZ1,S0} and references therein for
%homogenization of integral equations (or jump processes) with periodic coefficients.

%As mentioned above, known results concerning
%stochastic homogenization of stable-like Dirichlet forms in
%a  one-parameter  stationary  ergodic environment
%require the coefficients enjoying very special forms
%(for examples, the product form). The contribution of this paper is to
%study the homogenization problem systematically
%for symmetric non-local operators in a one-parameter stationary
%ergodic environment
%under more general settings,
%where the corresponding random coefficients can be degenerate and unbounded.

\subsection{Framework}
 Let $d\geq 1$ and suppose that $(\Omega, \mathcal{F}, \Pp)$ is a probability space
 endowed with a measurable group of transformations $\tau_x:\Omega \to \Omega,\ x\in \R^d$ such that
\begin{itemize}
\item [(a)] $\tau_0={\rm id}$, where ${\rm id}$ denotes the identity map on $\Omega$;

\item [(b)] $\tau_x \circ \tau_x=\tau_{x+y}$  for every $x, y\in \R^d$.
\end{itemize}
We also assume that $\{\tau_x\}_{x\in \R^d}$ is stationary and ergodic, i.e., the following properties hold.
\begin{itemize}
\item[(i)]  $\Pp(\tau_x A)=\Pp(A)$ for all $A\in \mathcal{F}$ and $x\in \R^d$(stationary condition);

\item[(ii)] if $A\in \mathcal{F}$ and $\tau_xA=A$ for all $x\in \R^d$, then $\Pp(A)\in \{0,1\}$(ergodic condition);

\item[(iii)] the function $(x,\w)\mapsto \tau_x \w$ is $\mathscr{B}(\R^d)\times \mathcal{F}$-measurable(measurable condition).
\end{itemize}
Let $\HH:=L^2(\Omega;\Pp)$. With this group of transformations $\{\tau_x\}_{x\in \R^d}$, we define the translation operator $T_x: \HH\to \HH$
(along direction $x\in \R^d$ in environments) and derivative operator $D_k$, $1\le k \le d$ as follows.
\begin{equation*}
T_x F(\w):=F(\tau_{-x} \w),\ x\in \R^d,\ F\in \HH,
\end{equation*}
\begin{equation*}
D_j F(\w):=\lim_{\e\downarrow 0}\frac{T_{\e e_j}F(\w)-F(\w)}{\e},\ F\in \HH,\ 1\le j \le d
\end{equation*}
where $\{e_j\}_{1\le j \le d}$ is the canonical orthonormal basis of $\R^d$. By such definition and direct computation we have
(if $D_jF$ exists)
\begin{equation*}%\label{e1-3-1}
D_j F(\tau_x \w)=-\frac{\partial\tilde F(\cdot;\w) }{\partial x_j}(x),\ x\in \R^d,
\end{equation*}
where $\tilde F(x;\w):=F(\tau_x \w)$.

We set
\begin{equation*}
\D:=\{F\in L^\infty(\Omega;\Pp); D_jF\ {\rm exists\ and}\ D_j F,\ D_k(D_j F)\in L^\infty(\Omega;\Pp)\ {\rm for\ every}\ 1\le j,k \le d\}.
\end{equation*}
Then by stationary property it is easy to verify the integration by parts formula for $D_j$
\begin{equation}\label{e1-2a}
\Ee[D_j F G]=-\Ee[F D_j G],\ \ F,G\in \D.
\end{equation}
Since $T_x$ is a strongly continuous unitary semi-group on $\HH$, we have a spectral resolution of
$T_x$ as follows
\begin{equation}\label{e1-1a}
T_x=\int_{\R^d} e^{i \langle \xi, x\rangle }U(d\xi),\ x\in \R^d,
\end{equation}
where $U(d\xi)$ is associated projection(on $\HH$) valued measure, and
$\langle \xi,x\rangle$ denotes the canonical inner product on $\R^d$. Based on the expression \eqref{e1-1a} we have
\begin{equation*}
\|F\|_{\HH_1}:=\Ee[F^2]+\sum_{j=1}^d \Ee[|D_j F|^2]=\int_{\R^d}\left(1+|\xi|^2\right)\Ee[U(d\xi)F\cdot F].
\end{equation*}
According to integration by parts formula \eqref{e1-2a}, we define $\HH_1:=\overline{\D}^{\|\cdot\|_{\HH_1}}$ as
the closed extension of $\D$ under $\|\cdot\|_{\HH_1}$ norm defined above. By this we know immediately $D_j F$
is still well defined for every $F\in \HH_1$ and the integration by parts formula \eqref{e1-2a} is true for
every $F,G\in \HH_1$.

Meanwhile for every $\beta\in (0,2)$, let
\begin{equation*}
\begin{split}
\LL_\beta F(\w)&:={\rm p.v.}\int_{\R^d}\left(F(\tau_{-z} \w)-F(\w)\right)\frac{1}{|z|^{d+\beta}}dz\\
&=\lim_{\e \downarrow 0}\int_{\{z\in \R^d;|z|>\e\}}\left(F(\tau_{-z} \w)-F(\w)-\sum_{j=1}^d D_jF(\w)z_j\right)\frac{1}{|z|^{d+\beta}}dz,
\ F\in \D.
\end{split}
\end{equation*}
where we use the notation $z=(z_1,\cdots,z_d)$ with $z_j$ denoting the $j$-th coordinate for $z\in \R^d$. By
mean-value inequality it is not difficult to verify that $\LL_\beta F$ is well defined  and $\LL_\beta F(\w)\in L^\infty(\Omega;\Pp)$ for every $F\in \D$.

Recall that $\Delta^{\beta/2}f(x):={\rm p.v.}\int_{\R^d}\left(f(x+z)-f(x)\right)\frac{C_0(d,\beta)}{|z|^{d+\beta}}dz$, $f\in C_b^2(\R^d)$
is the fractional Laplacian operator with order $\beta/2$, then by direct computation it is easy to see that
\begin{equation}\label{e1-3a}
\LL_\beta F(\tau_x \w)=C_0(d,\beta)^{-1}\Delta^{\beta/2}\tilde F(\cdot;\w)(x),\ \ x\in \R^d,\ F \in \D,
\end{equation}
where $\tilde F(x;\w):=F(\tau_x \w)$.

Applying the stationary property of $\tau_x$ again we have
\begin{equation}\label{e1-4}
\Ee[\LL_{\beta}F \cdot G]=\Ee[F\cdot \LL_{\beta}G]=-\frac{1}{2}\Ee\left[\int_{\R^d}\frac{(F(\tau_{-z}\w)-F(\w))(G(\tau_{-z}\w)-G(\w))}{|z|^{d+\beta}}dz\right]
,\ F,G\in \D.
\end{equation}

By \eqref{e1-1a} again we have
\begin{equation*}
\begin{split}
\|F\|_{\HH_{\beta/2}}&:=-2\Ee[\LL_{\beta}F \cdot F]+\Ee[F^2]=
\Ee\left[\int_{\R^d}\frac{(F(\tau_{-z}\w)-F(\w))^2}{|z|^{d+\beta}}dz\right]+\Ee[F^2]\\
&=-2\Ee\left[\left(\int_{\R^d}\frac{T_{z}F(\w)-F(\w)}{|z|^{d+\beta}}dz\right)\cdot F(\w)\right]+\Ee[F^2]\\
&=-2\Ee\left[\left(\int_{\R^d}\frac{e^{i\langle \xi, z\rangle}-1}{|z|^{d+\beta}}dz\right)\left(U(d\xi)F\cdot F\right)\right]+\Ee[F^2]\\
&=\int_{\R^d}\left(1+C_1(d,\beta)|\xi|^\beta\right)\Ee[U(d\xi)F\cdot F],\ F\in \D,
\end{split}
\end{equation*}
where we use the fact that
\begin{align*}
-2\int_{\R^d}\frac{e^{i\langle \xi, z\rangle}-1}{|z|^{d+\beta}}dz=C_1(d,\beta)|\xi|^\beta,\ \ \xi\in \R^d.
\end{align*}
for a positive constant $C_1(d,\beta)$.

Based on integration by parts formula \eqref{e1-4}, we can define $\HH_{\beta/2}:=\overline{\D}^{\|\cdot\|_{\HH_{\beta/2}}}$ as
the closed extension of $\D$ under $\|\cdot\|_{\HH_{\beta/2}}$ norm defined above.
Then $\left(\HH_{\beta/2},\|\cdot\|_{\beta/2}\right)$ is a Hilbert space. By the expression of $\|\cdot\|_{\HH_{\beta/2}}$ we know immediately that
\begin{equation*}
\HH_{1}\subset \HH_{\beta/2},\ \|F\|_{\HH_{\beta/2}}\le c_0\|F\|_{\HH_1},\ \ \forall\ \beta\in (0,2).
\end{equation*}

Moreover, as before for every $F\in \HH_{\beta/2}$, set $\tilde F(x;\w):=F(\tau_x \w)$, $x\in \R^d$, $\w\in \Omega$.
Applying standard approximation procedure we have $\tilde F(\cdot,\w)\in W_{loc}^{\beta/2}(\R^d)$ and
for every bounded subset $O\subset \R^d$,
\begin{equation}\label{e1-5a}
\begin{split}
\quad \Ee\left[\int_O \int_{\R^d}\frac{(\tilde F(x+z;\w)-\tilde F(x;\w))^2}{|z|^{d+\beta}}dzdx\right]
&=\int_O\Ee\left[\int_{\R^d}\frac{(F(\tau_z (\tau_x\w))-F(\tau_x\w))^2}{|z|^{d+\beta}}dz\right]dx\\
&=\int_O \Ee\left[\int_{\R^d}\frac{(F(\tau_z \w)-F(\w))^2}{|z|^{d+\beta}}dz\right]dx=\|F\|_{\HH_{\beta/2}}\cdot |O|,
\end{split}
\end{equation}
where the third equality is due to stationary property of $\tau_x$ and the change of variable $\tilde z=-z$. We refer the reader
to \cite{PV} for detailed introduction on the contents above.

\medskip

 Throughout this paper, let $\alpha\in (0,2)$ and assume that
\begin{assumption}\label{a1-1}
$\tilde \mu: \Omega\rightarrow (0,\infty)$ is
 a positive random variable such that
\begin{align*}
\inf_{\w\in \Omega}\tilde \mu(\w)\ge \theta_0>0,\ \ \ \Ee[|\tilde \mu|^2]<+\infty
\end{align*}
for some positive constant $\theta_0$.
\end{assumption}

We define $\mu:\R^d\times \Omega \to [0,\infty)$ as
\begin{equation*}
\mu(x;\w):=\tilde \mu(\tau_x \w),\ \ x\in \R^d,\ \w\in \Omega.
\end{equation*}
And we make the following assumption on $\mu$

\begin{equation}\label{e1-1}
x\mapsto \int (1\wedge |z|^2)\frac{\mu(x;\w)\mu(x+z;\w)}{|z|^{d+\alpha}}\,dz \in L^1_{loc}(\R^d;dx),
\quad \mbox{$\Pp$-a.e.}
\end{equation}

For a.s. $\w\in \Omega$, we define a (symmetric) bilinear form $\EE_0^\w$ as follows
\begin{equation}\label{e1-2}
\EE_0^\w(f,g):=\frac{1}{2}
\iint_{\R^d\times\R^d\setminus \Delta}
(f(x)-f(y))(g(x)-g(y))\frac{\mu(x;\w)\mu(y;\w)}{|x-y|^{d+\alpha}}
\,dx\,dy,\quad f,g\in C_c^\infty(\R^d),
\end{equation}
where $\Delta:=\{(x,x)\in \R^d\}$ is the diagonal of $\R^d \times \R^d$, and $C_c^\infty(\R^d)$ denotes the collection
of all $C^\infty$ functions on $\R^d$ with compact supports. In particular,
under assumption
 \eqref{e1-1}, it is easy to verify that $\EE_0^\w(f,f)<\infty$ for all $f\in C_c^1 (\R^d)$.
It is well known that (see e.g. \cite{FOT}) $(\EE_0^\w,\FF_0^\w)$ is a symmetric regular
Dirichlet form on $L^2(\R^d; dx)$ with $\FF_0^\w$
 being the closure of $C_c^1(\R^d)$ with respect to the norm $\|\cdot\|_{\E_0^\w}:=$ $(\E_0^\w(\cdot,\cdot)+\|\cdot\|_{L^2(\R^d; dx)}^2)^{1/2}$.
In particular, $\FF^\w_0$ is a Hilbert space endowed with the norm $\|\cdot\|_{\EE_0^\w}$.
Similarly, for every $\e\in (0,1)$, we can define the scaling Dirichlet form $(\E^{\e,\w}_{0},\FF_0^{\e,\w})$ as follows.
\begin{equation}\label{e1-3}
\EE_0^{\e,\w}(f,g):=\frac{1}{2}
\iint_{\R^d\times\R^d\setminus \Delta}
(f(x)-f(y))(g(x)-g(y))\frac{\mu\left(\frac{x}{\e};\w\right)\mu\left(\frac{y}{\e};\w\right)}{|x-y|^{d+\alpha}}
\,dx\,dy,\quad f,g\in \FF_0^{\e,\w},
\end{equation}
and $\FF_0^{\e,\w}$ is a Hilbert space endowed with the norm $\|\cdot\|_{\EE_0^{\e,\w}}:=$
$(\E_0^{\e,\w}(\cdot,\cdot)+\|\cdot\|_{L^2(\R^d; dx)}^2)^{1/2}$.

Let $\{X_t^\w\}_{t\ge 0}$ be the Hunt process associated with $(\EE_0^\w,\FF_0^\w)$(note that $(\EE_0^\w,\FF_0^\w)$ is regular), then
the scaled process $\{X_t^{\e,\w}\}_{t\ge 0}=\{\e X^{\w}_{\e^{-\alpha}t}\}_{t\ge 0}$ is the Hunt process related to
$(\EE_0^{\e,\w},\FF_0^{\e,\w})$.

Meanwhile, we assume that
\begin{assumption}\label{a1-2}
$\tilde b=(\tilde b_1,\tilde b_2,\cdots, \tilde b_d):\Omega \to \R^d$ is a random vector
and there exists a system of stream functions(for $\tilde b$) $\tilde H_{jl}:\Omega \to \R$, $1\le j,l\le d$ such that
the following properties hold.
\begin{itemize}
\item [(i)] $\tilde H_{jl}\in \HH_{1}$ for every $1\le j,l \le d$;

\item [(ii)] $\tilde H_{jl}=-\tilde H_{lj}$ for every $1\le j,l \le d$;

\item [(iii)] $\tilde b_j=\sum_{l=1}^d D_l \tilde H_{jl}$ for every $1\le j \le d$;

\item [(iv)] The function $x\mapsto D_l \tilde H_{jl}(\tau_x \w)$ is continuous for every $1\le l \le d$ and a.s. $\w\in \Omega$.
\end{itemize}
\end{assumption}
As before given $\tilde b:\Omega \to \R^d$ we define $b:\R^d\times \Omega \to \R^d$ by
\begin{equation*}
b(x;\w):=\tilde b(\tau_x \w)=\left(\tilde b_1(\tau_x \w), \tilde b_2(\tau_x \w),\cdots, \tilde b_d(\tau_x \w)\right),\ \ x\in \R^d,\ \w\in \Omega.
\end{equation*}
According to Assumption \ref{a1-2} it is easy to verify that
\begin{equation}\label{e1-5}
{\rm div}b(\cdot;\w)(x)=-\sum_{j=1}^d \frac{\partial b_j(x;\w)}{\partial x_j}=\sum_{j=1}^d D_j \tilde b_j(\tau_x \w)=
\sum_{j,l=1}^d D_jD_l \tilde H_{jl}(\tau_x\w)=0,
\end{equation}
where the last equality is in the distribution sense.

With $\left(\EE^\w_0,\FF^\w_0\right)$ and vector fields $b$ given above, we define a bilinear non-symmetric form
(for a.s. $\w\in \Omega$) as follows,
\begin{equation*}
\EE^\w(f,g):=\EE_0^\w(f,g)+\sum_{j=1}^d\int_{\R^d}b_j(x;\w)\frac{\partial f(x)}{\partial x_j}g(x)dx,\ \ \ f,g\in C_c^1(\R^d).
\end{equation*}
Using divergence free condition \eqref{e1-5}, definition of $(\EE^\w_0,\FF^\w_0)$, integration by parts formula
and standard approximation arguments, $\EE^\w$ is still well defined on $\FF_0^\w\times C_c^1(\R^d)$ as follows.
\begin{equation*}
\EE^\w(f,g)=\EE_0^\w(f,g)-\sum_{j=1}^d\int_{\R^d}b_j(x;\w)\frac{\partial g(x)}{\partial x_j}f(x)dx,\ \ \ f\in \FF_0^\w,\ g\in C_c^1(\R^d).
\end{equation*}
Then for every $\e\in (0,1)$ we define the scaled bi-linear form
\begin{equation*}
\EE^{\e,\w}(f,g)=\EE_0^{\e,\w}(f,g)-\e^{-(\alpha-1)}\sum_{j=1}^d\int_{\R^d}b_j\left(\frac{x}{\e};\w\right)\frac{\partial g(x)}{\partial x_j}f(x)dx,\ \ \ f\in
\FF_0^{\e,\w},\ g\in C_c^1(\R^d).
\end{equation*}
As before, if $\{Y_t^{\w}\}_{\ge 0}$ is the process associated with $\EE^{\w}$
(some extra condition is required to ensure the existence of $\{Y_t^\w\}_{t\ge 0}$, such as the uniform boundedness of $\mu(x;\w)$ and $b(x;\w)$), then
$\{Y_t^{\e,\w}\}_{t\ge 0}=\{\e Y_{\e^{-\alpha}t}^{\w}\}_{t\ge 0}$ is the process related to
$\EE^{\e,\w}$.

As shown in Proposition \ref{t1-1} below, for every $f\in C_c^1(\R^d)$, $\lambda>0$, $\e\in (0,1)$ and a.s. $\w\in \Omega$,
there exists a (weak) solution $u_{\lambda,f}^{\e,\w}\in \FF_0^{\e,\w}$ of the following resolvent equation associated with $\EE^{\e,\w}$,
\begin{equation}\label{e1-6}
\EE^{\e,\w} (u_{\lambda,f}^{\e,\w},g)+\lambda\int_{\R^d}u_{\lambda,f}^{\e,\w}(x)g(x)dx=\int_{\R^d}f(x)g(x)dx,\ \forall\ g\in C_c^1(\R^d).
\end{equation}

Now we will give the main theorem of this paper concerning the homogenization for scaled equation \eqref{e1-6}.

\begin{theorem}\label{t1-2}
For every $f\in C_c^1(\R^d)$, $\lambda>0$ and a.s. $\w\in \Omega$, let $u_{\lambda,f}^{\e,\w}$ be the solution
of \eqref{e1-6} obtained in Proposition \ref{t1-1}, then we have
\begin{equation}\label{t1-2-1}
\lim_{\e \to 0}\int_{O}|u_{\lambda,f}^{\e,\w}(x)-\bar u_{\lambda,f}(x)|^2dx=0,\ {\rm for\ every\ bounded\ subset\ } O\subset \R^d.
\end{equation}
Here $\bar u_{\lambda,f}$ is the unique solution in $L^2(\R^d)$ of the following equation
\begin{equation}\label{t1-2-2}
\lambda \bar u_{\lambda,f}(x)-\bar L \bar u_{\lambda,f}(x)=f(x),\ \ x\in \R^d,
\end{equation}
with
$$\bar L f(x):={\rm p.v.}\int_{\R^d}\left(f(x+z)-f(x)\right)\frac{\Ee[\mu]^2}{|z|^{d+\alpha}}dz,\ \ f\in C_c^2(\R^d).$$
\end{theorem}

\begin{remark}
For simplicity of the notation, we only consider the case that the jump kernel of $\EE^\w_0$
is with product form defined by \eqref{e1-3}. By the proof, the main theorem still holds when
the jump kernel of $\EE^\w_0$ has more general form as that in \cite{CCKW,CCKW1}.
\end{remark}

\section{Existence and uniqueness for the solution of \eqref{e1-6}}

\begin{proposition}\label{t1-1}
For every $f\in C_c^1(\R^d)$, $\lambda>0$, $\e\in (0,1)$ and a.s. $\w\in \Omega$, there exists a $u_{\lambda,f}^{\e,\w}\in \FF_0^{\e,\w}$ such that
\eqref{e1-6} holds,
%\begin{equation}\label{t1-1-1}
%\EE^{\e,\w} (u_{\lambda,f}^{\e,\w},g)+\lambda\int_{\R^d}u_{\lambda,f}^{\e,\w}(x)g(x)dx=\int_{\R^d}f(x)g(x)dx,\ \forall\ g\in C_c^1(\R^d),
%\end{equation}
and

\begin{equation}\label{t1-1-1a}
\sup_{\e\in (0,1)}\|u_{\lambda,f}^{\e,\w}\|_{L^2(\R^d)}<+\infty,\ \sup_{\e\in (0,1)}\|u_{\lambda,f}^{\e,\w}\|_{W^{\alpha/2,2}(\R^d)}<+\infty,\ \
\sup_{\e\in (0,1)}\EE_0^{\e,\w}\left(u_{\lambda,f}^{\e,\w},u_{\lambda,f}^{\e,\w}\right)<+\infty.
\end{equation}

Moreover if we assume $\alpha\in [1,2)$ and

\begin{equation}\label{t1-1-2a}
\sup_{\w\in \Omega}\left(|\tilde b(\w)|+\int_{\R^d}\frac{|b(\tau_z \w)-b(\w)|^2}{|z|^{d+2-\alpha}}dz\right)<+\infty,
\end{equation}
then $u_{\lambda,f}^{\e,\w}$ above is the unique solution of \eqref{e1-6} in $\FF_0^{\e,\w}$.
\end{proposition}

\begin{proof}
%[Proof of Theorem $\ref{t1-1}$]

{\bf Step 1}  By Assumption \ref{a1-1} we know immediately that $\FF^{\e,\w}_0 \subset W^{\alpha/2,2}(\R^d)$ and
\begin{equation}\label{t1-1-2}
\begin{split}
\|g\|_{\EE_0^{\e,\w}}&\ge c_1
\left(\int_{\R^d}\int_{\R^d}\frac{(g(x+z)-g(x))^2}{|z|^{d+\alpha}}dzdx+\int_{\R^d}|g(x)|^2dx\right)=
\|g\|_{W^{\alpha/2,2}(\R^d)},\ g\in \FF^{\e,\w}_0.
\end{split}
\end{equation}

For every $\theta>0$ and integer $k\in \mathbb{N_+}$, we define a bi-linear form $\EE^{\e,\w}_{\theta,k}:C_c^1(\R^d)\times C_c^1(\R^d)\to \R$ by
\begin{equation*}
\EE^{\e,\w}_{\theta,k}(g,h):=\EE^{\e,\w}_0(g,h)+\theta\sum_{j=1}^d\int_{\R^d}\frac{\partial g(x)}{\partial x_j}\frac{\partial h(x)}{\partial x_j}dx+
\e^{-(\alpha-1)}\sum_{j=1}^d \int_{\R^d}b_{j}^k\left(\frac{x}{\e};\w\right)\frac{\partial g(x)}{\partial x_j}h(x)dx,\ f,g\in C_c^1(\R^d),
\end{equation*}
where $b_j^k(x;\w):=\tilde b_j^k(\tau_x \w)$, $\tilde b_j^k(\w):=\sum_{l=1}^d D_l \left(\eta_k\left(\tilde H_{jl}(\cdot)\right)\right)(\w)$,
and $\eta_k:\R \to (0,+\infty)$ is a function such that
\begin{align*}
\eta_k(s)=
\begin{cases}
s,\ \ &\ |s|\le k,\\
\in (-k,k),\ \ &\ k<|s|<2k,\\
0,\ \ &\ |s|\ge 2k,
\end{cases}
\end{align*}
$\sup_{k\ge 1}\sup_{s\in \R}|\eta_k'(s)|\le 2$ and $\eta_k(s)=-\eta_k(-s)$ for every $s\in \R$.

Then by definition and anti-symmetric property of $\{\tilde H_{jl}\}_{1\le j,l\le d}$(also the anti-symmetry of
$\{\eta_k(\tilde H_{jl})\}_{1\le j,l\le d}$) we have
for every $f,g\in C_c^1(\R^d)$,
\begin{equation*}
\EE^{\e,\w}_{\theta,k}(g,h):=\EE^{\e,\w}_0(g,h)+\theta\sum_{j=1}^d\int_{\R^d}\frac{\partial g(x)}{\partial x_j}\frac{\partial h(x)}{\partial x_j}dx-
\e^{2-\alpha}\sum_{j,l=1}^d \int_{\R^d}\eta_k\left(H_{jl}\left(\frac{x}{\e};\w\right)\right)\frac{\partial g(x)}{\partial x_j}\frac{\partial h(x)}{\partial x_l}dx,
\end{equation*}
by this it is easy to verify that
\begin{equation}\label{t1-1-3}
\begin{split}
|\EE^{\e,\w}_{\theta,k}(g,h)|&\le c_2(\e,\theta,k)\left(\|g\|_{\EE^{\e,\w}_0}+\|g\|_{W^{1,2}(\R^d)}\right)
\cdot\left(\|h\|_{\EE^{\e,\w}_0}+\|h\|_{W^{1,2}(\R^d)}\right)\\
&=:c_2(\e,\theta,k)\|g\|_{\EE^{\e,\w}_1}\|h\|_{\EE^{\e,\w}_1},\ \ f,g\in C_c^1(\R^d),
\end{split}
\end{equation}
where $c_2(\e,\theta,k)$ is a positive constant which may depend on $\e$, $\theta$, $k$, and
$\|g\|_{\EE^{\e,\w}_1}:=\|g\|_{\EE^{\e,\w}_0}+\|g\|_{W^{1,2}(\R^d)}$.

Let $\mathbb{H}_1:=\overline{C_c^1(\R^d)}^{\|\cdot\|_{\EE^{\e,\w}_1}}$ be the closed extension of
$C_c^1(\R^d)$ under the norm $\|\cdot\|_{\EE^{\e,\w}_1}$ defined above, clearly we have
$\mathbb{H}_1\subset W^{1,2}(\R^d)$. Hence according to \eqref{t1-1-3} we can extend
$\EE^{\e,\w}_{\theta,k}$ to a continuous bilinear form on $\mathbb{H}_1\times \mathbb{H}_1$.

Meanwhile by anti-symmetry of $\{\tilde H_{jl}\}_{1\le j,l\le d}$ and the fact $\eta(s)=-\eta(-s)$ it holds that
\begin{align}\label{t1-1-4}
\sum_{j,l=1}^d\int_{\R^d}\eta_k\left(H_{jl}^k\left(\frac{x}{\e};\w\right)\right)\frac{\partial g(x)}{\partial x_j}\frac{\partial g(x)}{\partial x_l}dx=0,
\ \ \forall\ g\in \mathbb{H}_1,
\end{align}
which implies that for every $\lambda>0$,
\begin{equation*}
\|g\|_{L^2(\R^d)}^2+\EE_{\theta,k}^{\e,\w}(g,g)\ge c_3(\lambda,\theta,k)\|g\|_{\EE^\w_1}^2,\ \ g\in \mathbb{H}_1.
\end{equation*}
Therefore according to the Lax-Milgram theorem for every $f\in C_c^1(\R^d)$, there exists a unique $u_{\lambda,f}^{\e,\theta,k,\w}\in \mathbb{H}_1$ such that
\begin{equation}\label{t1-1-4a}
\begin{split}
&\quad \EE^{\e,\w}_0(u_{\lambda,f}^{\e,\theta,k,\w},g)+\theta\sum_{j=1}^d\int_{\R^d}\frac{\partial u_{\lambda,f}^{\e,\theta,k,\w}(x)}{\partial x_j}
\frac{\partial g(x)}{\partial x_j}\\
&-\e^{-(\alpha-1)}\sum_{j=1}^d \int_{\R^d}b_{j}^k\left(\frac{x}{\e};\w\right)\frac{\partial g(x)}{\partial x_j}
u_{\lambda,f}^{\e,\theta,k,\w}(x)dx
+\lambda\int_{\R^d}u_{\lambda,f}^{\e,\theta,k,\w}(x)g(x)dx\\
&=\int_{\R^d}f(x)g(x)dx,\ \forall\ g\in \mathbb{H}_1.
\end{split}
\end{equation}
So taking $g=u_{\lambda,f}^{\e,\theta,k,\w}$ in the above equation and using \eqref{t1-1-4} we have
\begin{align*}
&\quad \EE^{\e,\w}_0(u_{\lambda,f}^{\e,\theta,k,\w},u_{\lambda,f}^{\e,\theta,k,\w})+\theta\sum_{j=1}^d\int_{\R^d}
\left(\frac{\partial u_{\lambda,f}^{\e,\theta,k,\w}(x)}{\partial x_j}\right)^2dx+\lambda\int_{\R^d}|u_{\lambda,f}^{\e,\theta,k,\w}(x)|^2dx\\
&=\int_{\R^d}f(x)u_{\lambda,f}^{\e,\theta,k,\w}(x)dx
\le \frac{\lambda}{2}\int_{\R^d}|u_{\lambda,f}^{\e,\theta,k,\w}(x)|^2dx+(2\lambda)^{-1}\int_{\R^d}|f(x)|^2dx.
\end{align*}
By this and \eqref{t1-1-2} we obtain immediately that
\begin{align}\label{t1-1-6a}
\sup_{\e,\theta\in (0,1),k\ge 1}\|u_{\lambda,f}^{\e,\theta,k,\w}\|_{W^{\alpha/2,2}(\R^d)}\le c_1^{-1}\sup_{\e,\theta\in (0,1),k\ge 1}\EE^\w_0(u_{\lambda,f}^{\e,\theta,k,\w},u_{\lambda,f}^{\e,\theta,k,\w})\le c_4(\lambda)<\infty,
\end{align}
\begin{align}\label{t1-1-8}
\sup_{\e,\theta\in (0,1),k\ge 1}\int_{\R^d}|u_{\lambda,f}^{\e,\theta,k,\w}(x)|^2dx\le \lambda^{-2}\int_{\R^d}|f(x)|^2dx,
\end{align}
and for every $\theta\in (0,1)$,
\begin{align}\label{t1-1-5}
\sup_{\e\in (0,1),k\ge 1}\sum_{j=1}^d\int_{\R^d}
\left(\frac{\partial u_{\lambda,f}^{\e,\theta,k,\w}(x)}{\partial x_j}\right)^2dx\le c_4(\lambda)\theta^{-1}.
\end{align}
Hence for every fixed $\e\in (0,1)$, $\{u_{\lambda,f}^{\e,\theta,k,\w}\}_{\theta\in (0,1),k\ge 1}$ is weakly compact in $\left(\FF^{\e,\w}_0,\|\cdot\|_{\EE^{\e,\w}_0}\right)$ and
$W^{\alpha/2,2}(\R^d)$, so there exists a subsequence $\{u_{\lambda,f}^{\e,\theta_m,k_m,\w}\}$ and $u_{\lambda,f}^{\e,\w}\in \FF^{\e,\w}_0$ such that
%%$u_{\lambda,f}^{\e,\theta_m,k_m,\w}$ converges weakly to $u_{\lambda,f}^{\e,\w}$ in $\left(\FF^\w_0,\|\cdot\|_{\EE^\w_0}\right)$(as $m \to \infty$) and $W^{\alpha/2,2}(\R^d)$,
\begin{align*}
\lim_{m \to \infty}\EE^{\e,\w}_0(u_{\lambda,f}^{\e,\theta_m,k_m,\w},g)=\EE^{\e,\w}_0(u_{\lambda,f}^{\e,\w},g),\ \forall\ g\in C_c^1(\R^d),
\end{align*}
and
\begin{equation}\label{t1-1-5a}
\lim_{m \to \infty}\int_{O}\left|u_{\lambda,f}^{\e,\theta_m,k_m,\w}(x)-u_{\lambda,f}^{\e,\w}(x)\right|^2dx=0,\ {\rm for\ any\ bounded\ subset}\ O\subset \R^d.
\end{equation}
Using \eqref{t1-1-5} it holds that
\begin{align*}
\lim_{m \to \infty}\theta_m\sum_{j=1}^d\left|\int_{\R^d}\frac{\partial u_{\lambda,f}^{\e,\theta_m,k_m,\w}(x)}{\partial x_j}\frac{\partial g(x)}{\partial x_j}dx\right|
&\le \lim_{m \to \infty}\theta_m\sum_{j=1}^d \left(\left\|\frac{\partial u_{\lambda,f}^{\e,\theta_m,k_m,\w}}{\partial x_j}\right\|_{L^2(\R^d)}\cdot
\left\|\frac{\partial g}{\partial x_j}\right\|_{L^2(\R^d)}\right)\\
&\le c_5\lim_{m \to \infty}\theta_m^{1/2}=0.
\end{align*}
For every $g\in C_c^1(\R^d)$, let $O\subset \R^d$ be the support of $g$, since $O$ is bounded, by definition of $b_j^k(x;\w)$(note that $|\eta_k'(s)|\le 2$)
we know for every fixed $\e\in (0,1)$ (note that by Assumption \ref{a1-2} (iv), $x\mapsto D_l\tilde H_{jl}(\tau_x \w)$ is continuous)
\begin{align*}
\sup_{k\ge 1, x\in O}|b_j^k\left(\frac{x}{\e};\w\right)|\le c_6\sup_{x\in \frac{O}{\e}}\sup_{1\le j,l\le d}\left|\frac{H_{jl}(x;\w)}{\partial x_l}\right|<c_7(\e,\w)<+\infty,\ \ \w\in \Omega.
\end{align*}
According to this, \eqref{t1-1-5a} and dominated convergence theorem yields that
\begin{align*}
\lim_{m \to \infty}\int_{\R^d}b_j^{k_m}\left(\frac{x}{\e};\w\right)\frac{\partial g(x)}{\partial x_j}u_{\lambda,f}^{\e,\theta_m,k_m,\w}(x)dx&=
\lim_{m \to \infty}\int_{\R^d}b_j^{k_m}\left(\frac{x}{\e};\w\right)\frac{\partial g(x)}{\partial x_j}u_{\lambda,f}^{\e,\w}(x)dx\\
&=\int_{\R^d}b_j\left(\frac{x}{\e};\w\right)\frac{\partial g(x)}{\partial x_j}u_{\lambda,f}^{\e,\w}(x)dx,\ \forall\ g\in C_c^1(\R^d).
\end{align*}

Therefore putting $u_{\lambda,f}^{\e,\theta_m,k_m,\w}$ in equation \eqref{t1-1-4a} and letting $m \to \infty$ we arrive at
for every $g\in C_c^1(\R^d)$,
\begin{equation*}
\begin{split}
&\quad \EE^{\e,\w}_0(u_{\lambda,f}^{\e,\w},g)-\e^{-(\alpha-1)}\sum_{j=1}^d \int_{\R^d}b_{j}\left(\frac{x}{\e};\w\right)\frac{\partial g(x)}{\partial x_j}u_{\lambda,f}^{\e,\w}(x)dx
+\lambda\int_{\R^d}u_{\lambda,f}^{\e,\w}(x)g(x)dx=\int_{\R^d}f(x)g(x)dx.
\end{split}
\end{equation*}
Combining this with \eqref{t1-1-6a}, \eqref{t1-1-8} implies \eqref{e1-6} and \eqref{t1-1-1a}.

{\bf Step 2} Now we consider the uniqueness of the solution.
For every $g,h\in C_c^1(\R^d)$ and $1\le j \le d$, by Parseval's identity it holds that
\begin{equation*}
\begin{split}
\int_{\R^d} b_j\left(\frac{x}{\e};\w\right)\frac{\partial g(x)}{\partial x_j} h(x)dx&=
\int_{\R^d}\F\left(\frac{\partial g}{\partial x_j}\right)(\xi)
\F\left(b_j\left(\frac{\cdot}{\e};\w\right)h(\cdot)\right)(\xi)d\xi\\
&=i \int_{\R^d}\xi_j\F(g)(\xi)\F\left(b_j\left(\frac{\cdot}{\e};\w\right)h(\cdot)\right)(\xi)d\xi,
\end{split}
\end{equation*}
where $\F(g)(\xi):=\frac{1}{(2\pi)^d}\int_{\R^d}e^{i\langle \xi, x\rangle}g(x)dx$ denotes
the Fourier transform for a function $g\in L^1(\R^d)$.

By this we have
\begin{equation}\label{t1-1-7}
\begin{split}
&\quad\left|\int_{\R^d} b_j\left(\frac{x}{\e};\w\right)\frac{\partial g(x)}{\partial x_j} h(x)dx\right|\\
&\le
\int_{\R^d}|\xi_j|\left|\F(g)(\xi)\right|\left|\F\left(b_j\left(\frac{\cdot}{\e};\w\right)h(\cdot)\right)(\xi)\right|d\xi\\
&\le c_8\left(\int_{\R^d}|\xi|^\alpha\left|\F(g)(\xi)\right|^2dx\right)^{1/2}\cdot\left(\int_{\R^d}|\xi|^{2-\alpha}
\left|\F\left(b_j\left(\frac{\cdot}{\e};\w\right)h(\cdot)\right)(\xi)\right|^2 d\xi\right)^{1/2}\\
&=c_9\left(\int_{\R^d}\int_{\R^d}\frac{(g(x+z)-g(x))^2}{|z|^{d+\alpha}}dzdx\right)^{1/2}\cdot
\left(\int_{\R^d}\int_{\R^d}\frac{(b_j\left(\frac{x+z}{\e};\w\right)h(x+z)-b_j\left(\frac{x}{\e};\w\right)h(x))^2}{|z|^{d+2-\alpha}}dzdx\right)^{1/2}.
\end{split}
\end{equation}
Meanwhile according to \eqref{t1-1-2a} it holds that
\begin{align*}
&\quad\int_{\R^d}\int_{\R^d}\frac{\left(b_j\left(\frac{x+z}{\e};\w\right)h(x+z)-b_j\left(\frac{x}{\e};\w\right)h(x)\right)^2}{|z|^{d+2-\alpha}}dzdx\\
&\le 2\int_{\R^d}\int_{\R^d}\left|b_j\left(\frac{x+z}{\e};\w\right)\right|^2\frac{\left|h(x+z)-h(x)\right|^2}{|z|^{d+2-\alpha}}dzdx+
2\int_{\R^d}\int_{\R^d}|h(x)|^2\frac{\left|b_j\left(\frac{x+z}{\e};\w\right)-b_j\left(\frac{x}{\e};\w\right)\right|^2}{|z|^{d+2-\alpha}}dzdx\\
&=2\int_{\R^d}\int_{\R^d}\left|\tilde b_j(\tau_{\frac{x+z}{\e}}\w)\right|^2\frac{\left|h(x+z)-h(x)\right|^2}{|z|^{d+2-\alpha}}dzdx+
2\e^{-(2-\alpha)}\int_{\R^d}\int_{\R^d}|h(x)|^2
\frac{\left|\tilde b_j\left(\tau_z(\tau_{\frac{x}{\e}}\w)\right)-\tilde b_j\left(\tau_{\frac{x}{\e}};\w\right)\right|^2}{|z|^{d+2-\alpha}}dzdx\\
&\le c_{10}\left(\int_{\R^d}\int_{\R^d}\frac{\left|h(x+z)-h(x)\right|^2}{|z|^{d+2-\alpha}}dzdx+\e^{-(2-\alpha)}\int_{\R^d}|h(x)|^2dx\right).
\end{align*}
Also note that $2-\alpha\le \alpha$ when $\alpha\in [1,2)$, so we derive
\begin{align*}
&\quad\int_{\R^d}\int_{\R^d}\frac{\left|h(x+z)-h(x)\right|^2}{|z|^{d+2-\alpha}}dzdx\\
&=\int_{\R^d}\int_{\{|z|\le 1\}}\frac{\left|h(x+z)-h(x)\right|^2}{|z|^{d+2-\alpha}}dzdx+
\int_{\R^d}\int_{\{|z|> 1\}}\frac{\left|h(x+z)-h(x)\right|^2}{|z|^{d+2-\alpha}}dzdx\\
%&\le \int_{\R^d}\int_{\{|z|\le 1\}}\frac{\left|h(x+z)-h(x)\right|^2}{|z|^{d+\alpha}}dzdx+
%2\int_{\R^d}\int_{\{|z|>1\}}\frac{\left(|h(x)|^2+|h(x+z)|^2\right)}{|z|^{d+\alpha}}dzdx\\
&\le \int_{\R^d}\int_{\{|z|\le 1\}}\frac{\left|h(x+z)-h(x)\right|^2}{|z|^{d+\alpha}}dzdx+c_{11}\int_{\R^d}|h(x)|^2dx
\end{align*}
Putting all above estimates together into \eqref{t1-1-7} yields that
\begin{align*}
\left|\int_{\R^d} b_j\left(\frac{x}{\e},\w\right)\frac{\partial g(x)}{\partial x_j} h(x)dx\right|\le
c_{12}(\e)\|g\|_{W^{\alpha/2,2}(\R^d)}\|h\|_{W^{\alpha/2,2}(\R^d)},\ \ \forall\ h,g\in C_c^1(\R^d),
\end{align*}
where $c_{12}(\e)$ is a positive constant which may depend on $\e$.

By this, \eqref{t1-1-2} and the definition of $\EE^{\e,\w}$ we know immediately that
\begin{align*}
\EE^{\e,\w}(h,g)\le c_{13}(\e)\|g\|_{\EE^{\e,\w}_0}\|h\|_{\EE^{\e,\w}_0},\ \ \forall\ h,g\in C_c^1(\R^d).
\end{align*}
Based on the estimate we can extend $\EE^{\e,\w}$ to a continuous bilinear form on $\EE^{\e,\w}_0\times \EE^{\e,\w}_0$.

Hence applying \eqref{t1-1-4} and using standard approximation arguments we derive
\begin{align*}
c_{14}\|g\|_{\EE^{\e,\w}_0}^2\le \EE^{\e,\w}(g,g) \le c_{15}\|g\|_{\EE^{\e,\w}_0}^2,\ \ \forall\ g\in \EE^{\e,\w}_0,
\end{align*}
where the positive constants $c_{14}$, $c_{15}$ here are independent of $\e$.
By this we can apply the Lax-Milgram theorem to prove that there exists a unique solution $u_{\lambda,f}^{\e,\w}$ for
\eqref{e1-6} which satisfies \eqref{t1-1-1a}.
\end{proof}

\section{Proof of Main Theorem}
Now we are ready to prove Theorem \ref{t1-2}.
\begin{proof} [Proof of Theorem $\ref{t1-2}$]
By \eqref{t1-1-1a} we know $\{u_{\lambda,f}^{\e,\w}\}_{\e\in (0,1)}$ is weakly compact in
$W^{\alpha/2,2}(\R^d)$, hence it is compact in $L_{loc}^2(\R^d)$. Then it suffices to prove that for every
subsequence $\{u_{\lambda,f}^{\e_m,\w}\}_{m\ge 1}$ which is convergent in $L_{loc}^2(\R^d)$, it holds that
\begin{equation*}
\lim_{m \to \infty}\int_{\R^d}\left|u_{\lambda,f}^{\e_m,\w}-\bar u_{\lambda,f}(x)\right|^2dx=0, \ \ {\rm for\ every\ bounded\ subset}\ O\subset \R^d,
\end{equation*}
where $\bar u_{\lambda,f}$ is the unique solution of \eqref{t1-2-2}.

Suppose that $\{u_{\lambda,f}^{\e_m,\w}\}_{m\ge 1}$ is a (arbitrarily fixed) subsequence such that
\begin{equation}\label{t1-2-3}
\lim_{m \to \infty}\int_{O}\left|u_{\lambda,f}^{\e_m,\w}(x)-u_0^{\w}(x)\right|^2dx=0,\ \ {\rm for\ every\ bounded\ subset}\ O\subset \R^d,
\end{equation}
for some $u_0^{\w}\in L^2(\R^d)$. Now we are going to prove $u_0^{\w}=\bar u_{\lambda,f}$.

According to \eqref{e1-6} we obtain for every $g\in C_c^1(\R^d)$,
\begin{equation}\label{t1-2-4}
\EE_0^{\e_m,\w}(u_{\lambda,f}^{\e_m,\w},g)-
\e_m^{-(\alpha-1)}\sum_{j=1}^d\int_{\R^d}b_j\left(\frac{x}{\e},\w\right)\frac{\partial g(x)}{\partial x_j}u_{\lambda,f}^{\e_m,\w}(x)dx
+\lambda\int_{\R^d}u_{\lambda,f}^{\e_m,\w}(x)g(x)dx=\int_{\R^d}f(x)g(x)dx.
\end{equation}

{\bf Step 1} In this step, we are going to prove
\begin{equation}\label{t1-2-5}
\lim_{m \to \infty}\EE_0^{\e_m,\w}(u_{\lambda,f}^{\e_m,\w},g)=-\int_{\R^d}u_0^\w(x)\bar L g(x)dx,\ \forall\ g\in C_c^2(\R^d).
\end{equation}
The main method in this step is similar with that in \cite{CCKW1}, for convenience of the reader we will give a detailed proof here.
For every $\delta\in (0,1)$ we set
\begin{equation*}
\begin{split}
&\quad \EE_0^{\e_m,\w}(u_{\lambda,f}^{\e_m,\w},g)\\
&=\frac{1}{2}
\int_{\R^d}\left(\int_{\{|z|\le \delta\}}+\int_{\{\delta<|z|<\delta^{-1}\}}+\int_{\{|z|\ge \delta^{-1}\}}\right)
\left(u_{\lambda,f}^{\e_m,\w}(x+z)-u_{\lambda,f}^{\e_m,\w}(x)\right)\left(g(x+z)-g(x)\right)\\
&\cdot\frac{\mu\left(\frac{x}{\e_m};\w\right)\mu\left(\frac{x+z}{\e_m};\w\right)}{|z|^{d+\alpha}}dzdx
=:I_1^{m,\delta,\w}+I_2^{m,\delta,\w}+I_3^{m,\delta,\w}.
\end{split}
\end{equation*}
Suppose that ${\rm supp}g\subset B(0,R_0):=\{x\in \R^d; |x|\le R_0\}$ for some $R_0>0$. Apply Cauchy-Schwartz inequality
we derive that
\begin{align*}
|I_1^{m,\delta,\w}|&\le \left(\frac{1}{2}\int_{\R^d}\int_{\{|z|\le \delta\}}
\left(u_{\lambda,f}^{\e_m,\w}(x+z)-u_{\lambda,f}^{\e_m,\w}(x)\right)^2\frac{\mu\left(\frac{x}{\e_m};\w\right)\mu\left(\frac{x+z}{\e_m};\w\right)}{|z|^{d+\alpha}}dzdx\right)^{1/2}\\
&\quad \cdot \left(\frac{1}{2}\int_{\R^d}\int_{\{|z|\le \delta\}}
\left(g(x+z)-g(x)\right)^2\frac{\mu\left(\frac{x}{\e_m};\w\right)\mu\left(\frac{x+z}{\e_m};\w\right)}{|z|^{d+\alpha}}\right)^{1/2}\\
&\le \EE_0^{\e_m,\w}\left(u_{\lambda,f}^{\e_m,\w},u_{\lambda,f}^{\e_m,\w}\right)^{1/2}\cdot
\left(\frac{\|\nabla g\|_\infty^2}{2}\int_{\{|x|\le R_0+1\}}\int_{\{|z|\le \delta\}}
\left(\mu\left(\frac{x+z}{\e_m};\w\right)^2+\mu\left(\frac{x}{\e_m};\w\right)^2\right)\frac{|z|^2}{|z|^{d+\alpha}}dzdx\right)^{1/2}\\
&\le c_1(\w)
\left(\|\nabla g\|_\infty^2\int_{\{|x|\le R_0+2\}}\mu\left(\frac{x}{\e_m};\w\right)^2\left(\int_{\{|z|\le \delta\}}
\frac{|z|^2}{|z|^{d+\alpha}}dz\right)dx\right)^{1/2},
\end{align*}
where $c_1(\w):=\sup_{m\ge 1}\EE_0^{\e_m,\w}\left(u_{\lambda,f}^{\e_m,\w},u_{\lambda,f}^{\e_m,\w}\right)^{1/2}$
is a positive constant(may depend on $\w$) which is finite due to \eqref{t1-1-1a}.

By direct computation we obtain
\begin{align*}
\int_{\{|x|\le R_0+2\}}\int_{\{|z|\le \delta\}}
\mu\left(\frac{x}{\e_m};\w\right)^2\frac{|z|^2}{|z|^{d+\alpha}}dzdx&\le
c_2\delta^{2-\alpha}\cdot \int_{\{|x|\le R_0+2\}}\tilde \mu\left(\tau_\frac{x}{\e_m}\w\right)^2dx.
\end{align*}
Hence by Birkhoff ergodic theorem (see, for example, \cite[Proposition 2.1]{CCKW1} or \cite[Theorem 7.2]{JKO}) it holds that for every fixed $\delta\in (0,1)$,
\begin{align*}
\lim_{m \to \infty}\int_{\{|x|\le R_0+2\}}\int_{\{|z|\le \delta\}}
\mu\left(\frac{x}{\e_m};\w\right)^2\frac{|z|^2}{|z|^{d+\alpha}}dzdx \le c_3\Ee[\tilde \mu^2]\delta^{2-\alpha}.
\end{align*}

Similarly for $I_3^{m,\delta,\w}$ we obtain
\begin{align*}
|I_3^{m,\delta,\w}|&\le \left(\frac{1}{2}\int_{\R^d}\int_{\{|x-y|> \delta^{-1}\}}
\left(u_{\lambda,f}^{\e_m,\w}(y)-u_{\lambda,f}^{\e_m,\w}(x)\right)^2\frac{\mu\left(\frac{x}{\e_m};\w\right)\mu\left(\frac{y}{\e_m};\w\right)}{|x-y|^{d+\alpha}}dydx\right)^{1/2}\\
&\quad \cdot \left(\frac{1}{2}\int_{\R^d}\int_{\{|x-y|> \delta^{-1}\}}
\left(g(y)-g(x)\right)^2\frac{\mu\left(\frac{x}{\e_m};\w\right)\mu\left(\frac{y}{\e_m};\w\right)}{|x-y|^{d+\alpha}}dydx\right)^{1/2}.\\
%&\le \EE_0^{\e_m,\w}\left(u_{\lambda,f}^{\e_m,\w},u_{\lambda,f}^{\e_m,\w}\right)^{1/2}\cdot
%\left(\|g\|_\infty^2\int_{\{|x|\le R_0+1\}}\int_{\{|z|> \delta^{-1}\}}
%\frac{\left(\mu\left(\frac{x+z}{\e_m};\w\right)^2+\mu\left(\frac{x}{\e_m};\w\right)^2\right)}{|z|^{d+\alpha}}dzdx\right)^{1/2}\\
\end{align*}
For $\delta^{-1}>2(R_0+2)$, we know $\left(g(x)-g(y)\right)1_{\{|x-y|>\delta^{-1}\}}\neq 0$ only if
$x\in B(0,R_0)$ or $y\in B(0,R_0)$. Then using the symmetry of $x$ and $y$ we have
\begin{align*}
&\varlimsup_{m\to \infty} \iint_{\{|x-y|>\delta^{-1}\}}(g(y)- g(x))^2\frac{\mu\left(\frac{x}{\e_m};\w\right)\mu\left(\frac{y}{\e_m};\w\right)}{|x-y|^{d+\alpha}}\,dx\,dy
\\
%&\leq  2 \limsup_{\e \to 0}  \iint_{\{|x-y|>1/\eta\}}(g(y)- g(x))^2\frac{\mu\left(\frac{x}{\e_m};\w\right)\mu\left(\frac{y}{\e_m};\w\right)}{|x-y|^{d+\alpha}}\,dx\,dy \\
& \leq  4\|g\|_\infty^2\varlimsup_{m \to \infty}  \int_{B(0, R_0)} \left( \int_{\{|y-x|>\delta^{-1}\}} \frac{\tilde \mu\left(\frac{y}{\e_m};\w\right)}
{|y-x|^{d+\alpha}} dy \right)
 \mu\left(\frac{x}{\e_m};\w\right) \,dx \\
&\le c_2\varlimsup_{m \to \infty}\int_{\{|x|\le R_0+2\}}\mu\left(\frac{x}{\e_m};\w\right)^2\left(\int_{\{|x-y|> \delta^{-1}\}}
\frac{1}{|x-y|^{d+\alpha}}dy\right)dx\\
&\quad\quad +c_2\varlimsup_{m \to \infty}\int_{\{|x|\le R_0+2\}}\int_{\{|z|> \delta^{-1}\}}\frac{\mu\left(\frac{y}{\e_m};\w\right)^2}{|x-y|^{d+\alpha}}dydx,
\end{align*}
where the last step follows from Cauchy-Schwartz inequality.
Applying Birkhoff ergodic theorem  we get
\begin{align*}
\lim_{m \to \infty}\int_{\{|x|\le R_0+2\}}\mu\left(\frac{x}{\e_m};\w\right)^2\left(\int_{\{|x-y|> \delta^{-1}\}}
\frac{1}{|x-y|^{d+\alpha}}dy\right)dx\le c_4\Ee[\tilde \mu^2]\delta^{\alpha}.
\end{align*}
%By change of variable $y=x+z$,
For $\delta^{-1}\ge 2(R_0+2)$ we have
\begin{align*}
%\quad\int_{\{|x|\le R_0+2\}}\int_{\{|z|> \delta^{-1}\}}\frac{\mu\left(\frac{x+z}{\e_m};\w\right)^2}{|z|^{d+\alpha}}dzdx
\int_{\{|x|\le R_0+2\}}\int_{\{|x-y|> \delta^{-1}\}}\frac{\tilde \mu\left(\tau_{\frac{y}{\e_m}}\w\right)^2}{|x-y|^{d+\alpha}}dydx
&\le c_5\int_{\{|x|\le R_0+2\}}\int_{\{|y|> (2\delta)^{-1}\}}\frac{\tilde \mu\left(\tau_{\frac{y}{\e_m}}\w\right)^2}{|y|^{d+\alpha}}dydx\\
&=:c_5\sum_{k=0}^\infty |B(0,R_0+2)|\int_{U_k}\frac{\tilde \mu\left(\tau_{\frac{y}{\e_m}}\w\right)^2}{|y|^{d+\alpha}}dy\\
&\le c_6\delta^{d+\alpha}\sum_{k=0}^\infty 2^{-k(d+\alpha)}\e_m^{d}\int_{B\left(0,\frac{2^{k+1}(2\delta)^{-1}}{\e_m}\right)}\tilde \mu\left(\tau_{y}\w\right)^2 dy.
\end{align*}
where $U_k:=\{y\in \R^d; 2^k(2\delta)^{-1}<|y|\le 2^{k+1}(2\delta)^{-1}\}$ for every integer $k\ge 0$. According to the ergodic theorem again
there exists a constant $c_7(\w)$(which may depend on $\w$) such that
\begin{align*}
\sup_{m\ge 1}\e_m^d\int_{B(0,\e_m^{-1})} \tilde \mu(\tau_y \w)^2 dy\le c_7(\w)\Ee[\tilde \mu^2].
\end{align*}
Hence using this in above inequality yields
\begin{align*}
\sup_{m\ge 1}\int_{\{|x|\le R_0+2\}}\int_{\{|x-y|> \delta^{-1}\}}\frac{\mu\left(\frac{y}{\e_m};\w\right)^2}{|x-y|^{d+\alpha}}dydx
&\le c_8(\w)\delta^\alpha\Ee[\tilde \mu^2]\sum_{k=0}^\infty 2^{-k\alpha}\le c_9(\w)\delta^\alpha.
\end{align*}
Combining all above estimates we can verify directly that
\begin{align}\label{t1-2-6}
\lim_{\delta \downarrow 0}\varlimsup_{m \to \infty}\left(|I_1^{m,\delta,\w}|+|I_3^{m,\delta,\w}|\right)=0.
\end{align}

By the change of variable $\tilde z=-z$ and $\tilde x=x+z$ for the integral
$$\int_{\R^d}\int_{\{\delta<|z|<\delta^{-1}\}}u_{\lambda,f}^{\e_m,\w}\left(x+z;\w\right)\left(g(x+z)-g(x)\right)
\frac{\mu\left(\frac{x}{\e_m};\w\right)\mu\left(\frac{x+z}{\e_m};\w\right)}{|z|^{d+\alpha}}dzdx$$
we have
\begin{align*}
I_2^{m,\delta,\w}&=-\int_{\R^d}\int_{\{\delta<|z|<\delta^{-1}\}}
u_{\lambda,f}^{\e_m,\w}\left(x;\w\right)\left(g(x+z)-g(x)\right)
\frac{\mu\left(\frac{x}{\e_m};\w\right)\mu\left(\frac{x+z}{\e_m};\w\right)}{|z|^{d+\alpha}}dzdx,
\end{align*}
which implies that for each fixed $\delta\in (0,1)$,
\begin{align*}
&\quad \lim_{m \to \infty}\left|I_2^{m,\delta,\w}-\frac{1}{2}\int_{\R^d}\int_{\{\delta<|z|<\delta^{-1}\}}
\left(u_{0}^{\w}(x+z)-u_0^\w(x)\right)\left(g(x+z)-g(x)\right)
\frac{\mu\left(\frac{x}{\e_m};\w\right)\mu\left(\frac{x+z}{\e_m};\w\right)}{|z|^{d+\alpha}}dzdx\right|\\
&\le \lim_{m \to \infty}c_{10}\|g\|_\infty\delta^{-d-\alpha}\cdot
\left(\int_{\{|x|\le R_0+\delta^{-1}\}}\int_{\{\delta<|z|<\delta^{-1}\}}\left|u_{\lambda,f}^{\e_m,\w}\left(x\right)-u_0^\w(x)\right|^2dzdx\right)^{1/2}\\
&\cdot \left(\int_{\{|x|\le R_0+\delta^{-1}\}}\int_{\{\delta<|z|<\delta^{-1}\}}\mu\left(\frac{x}{\e_m};\w\right)^2
\mu\left(\frac{x+z}{\e_m};\w\right)^2 dzdx\right)^{1/2}\\
&\le \lim_{m \to \infty}c_{11}(\delta)
\left(\int_{\{|x|\le R_0+\delta^{-1}\}}\left|u_{\lambda,f}^{\e_m,\w}\left(x\right)-u_0^\w(x)\right|^2dx\right)^{1/2}\\
&\cdot \left(\int_{\{|x|\le R_0+2\delta^{-1}\}}\int_{\{|y|\le R_0+2\delta^{-1}\}}\tilde \mu\left(\tau_{\frac{x}{\e_m}}\w\right)^2
\tilde \mu\left(\tau_{\frac{y}{\e_m}}\w\right)^2 dxdy\right)^{1/2}\\
&\le \lim_{m \to \infty}c_{11}(\delta)\Ee[\tilde \mu^2]\left(\int_{\{|x|\le R_0+\delta^{-1}\}}\left|u_{\lambda,f}^{\e_m,\w}\left(x\right)-u_0^\w(x)\right|^2dx\right)^{1/2}=0,
\end{align*}
where in the third inequality we have used the ergodic theorem, and the last step is due to \eqref{t1-2-3}.

According to \cite[Lemma 3.1(ii)]{CCKW1} we have
\begin{align*}
&\quad \lim_{m \to \infty}\int_{\R^d}\int_{\{\delta<|z|<\delta^{-1}\}}
\left(u_{0}^{\w}(x+z)-u_0^\w(x)\right)\left(g(x+z)-g(x)\right)
\frac{\mu\left(\frac{x}{\e_m};\w\right)\mu\left(\frac{x+z}{\e_m};\w\right)}{|z|^{d+\alpha}}dzdx\\
&=\int_{\R^d}\int_{\{\delta<|z|<\delta^{-1}\}}
\left(u_{0}^{\w}(x+z)-u_0^\w(x)\right)\left(g(x+z)-g(x)\right)
\frac{\Ee[\tilde \mu]^2}{|z|^{d+\alpha}}dzdx.
\end{align*}
Therefore putting all these estimates together yields that
\begin{align*}
\lim_{m \to \infty}I_2^{m,\delta,\w}&=\frac{1}{2}\int_{\R^d}\int_{\{\delta<|z|<\delta^{-1}\}}
\left(u_{0}^{\w}(x+z)-u_0^\w(x)\right)\left(g(x+z)-g(x)\right)
\frac{\Ee[\tilde \mu]^2}{|z|^{d+\alpha}}dzdx\\
&=-\int_{\R^d}u_0^{\w}(x)\left(\int_{\{\delta<|z|<\delta^{-1}\}}\left(g(x+z)-g(x)\right)
\frac{\Ee[\tilde \mu]^2}{|z|^{d+\alpha}}dz\right)dx.
\end{align*}
According to this and \eqref{t1-2-6}, firstly letting $m \to \infty$, then $\delta \downarrow 0$ we
can prove \eqref{t1-2-5}.

{\bf Step 2} It remains to prove that
\begin{align}\label{t1-2-7}
\lim_{m \to \infty}\e_m^{-(\alpha-1)}\sum_{j=1}^d\int_{\R^d}b_j\left(\frac{x}{\e_m};\w\right)\frac{\partial g(x)}{\partial x_j}u_{\lambda,f}^{\e_m,\w}(x)dx=0,
\ \forall\ g\in C_c^2(\R^d).
\end{align}
Without loss of generality, we assume that ${\rm supp}g\subset B(0,R_0)$ for some $R_0>0$.

%When $\alpha\in (0,1)$, according to the fact $\tilde b_j \in L^2(\Omega;\Pp)$ and \eqref{t1-1-1a} and applying Cauchy-Schwartz inequality and
%ergodic theorem we obtain
%\begin{align*}
%&\quad\lim_{m \to \infty}\e_m^{-(\alpha-1)}\sum_{j=1}^d\left|\int_{\R^d}b_j\left(\frac{x}{\e_m},\w\right)\frac{\partial g(x)}{\partial x_j}u_{\lambda,f}^{\e_m,\w}(x)dx\right|\\
%&\le c_{12}\sup_{\e>0}\|u_{\lambda,f}^{\e,\w}\|_{L^2(\R^d)}\cdot\lim_{m \to \infty}\e_m^{1-\alpha}
%\sqrt{\int_{B(0,R_0)}\left|\tilde b_j\left(\tau_{\frac{x}{\e_m}}\w\right)\right|^2dx}=0.
%\end{align*}
%This means that \eqref{t1-2-7} is true for $\alpha\in (0,1)$.

%Now we consider the case $\alpha\in [1,2)$.
By Assumption \ref{a1-2} we have $H_{jl}(\cdot;\w)\in W^{1,2}_{loc}(\R^d)$ for $H_{jl}(x;\w):=\tilde H_{jl}(\tau_x \w)$.
%and $b_j\left(x;\w\right)=\sum_{l=1}^d \frac{\partial H_{jl}(\cdot;\w)}{\partial x_l}(x)$.
Since ${\rm supp}g\subset B(0,R_0)$, we choose a function $h\in C_c^2(\R^d)$ such that
$h(x)=1$ for every $x\in B(0,R_0)$ and ${\rm supp}h \subset B(0,2R_0)$. Based on this it holds that
\begin{align}\label{t1-2-7a}
\e_m^{-(\alpha-1)}\int_{\R^d}b_j\left(\frac{x}{\e},\w\right)\frac{\partial g(x)}{\partial x_j}u_{\lambda,f}^{\e_m,\w}(x)dx&=
\e_m^{2-\alpha}\sum_{l=1}^d\int_{\R^d}  \frac{\partial}{\partial x_l}\left( H_{jl}^{\e_m}\left(\cdot;\w\right)h(\cdot)\right)(x)
\frac{\partial g(x)}{\partial x_j}u_{\lambda,f}^{\e_m,\w}(x)dx,
\end{align}
where $H_{jl}^{\e_m}(x;\w):=H_{jl}\left(\frac{x}{\e_m};\w\right)$.

So by Parseval's identity we get for every $1\le j,l\le d$,
\begin{align*}
&\quad \e_m^{2-\alpha}\left|\int_{\R^d}\frac{\partial}{\partial x_l}\left( H_{jl}^{\e_m}\left(\cdot;\w\right)h(\cdot)\right)(x)
\frac{\partial g(x)}{\partial x_j}u_{\lambda,f}^{\e_m,\w}(x)dx\right|\\
&=\e_m^{2-\alpha}\left|\int_{\R^d}i \xi_l \F\left(H_{jl}^{\e_m}\left(\cdot;\w\right)h(\cdot)\right)(\xi)
\cdot\F\left(\frac{\partial g(\cdot)}{\partial x_j}u_{\lambda,f}^{\e_m,\w}(\cdot)\right)(\xi)d\xi\right|\\
&\le c_{13}\e_m^{2-\alpha}\left(\int_{\R^d}|\xi|^\alpha\left|\F\left(\frac{\partial g(\cdot)}{\partial x_j}
u_{\lambda,f}^{\e_m,\w}(\cdot)\right)(\xi)\right|^2d\xi\right)^{1/2}\cdot
\left(\int_{\R^d}|\xi|^{2-\alpha}\left|\F\left(H_{jl}^{\e_m}\left(\cdot;\w\right)h(\cdot)\right)(\xi)\right|^2d\xi\right)^{1/2}.
\end{align*}
Then we have
\begin{align*}
&\quad\int_{\R^d}|\xi|^\alpha \left|\F\left(\frac{\partial g(\cdot)}{\partial x_j}
u_{\lambda,f}^{\e_m,\w}(\cdot)\right)(\xi)\right|^2d\xi\\
&\le c_{14}\int_{\R^d}\int_{\R^d}\frac{\left|\frac{\partial g(x+z)}{\partial x_j}
u_{\lambda,f}^{\e_m,\w}(x+z)-\frac{\partial g(x)}{\partial x_j}
u_{\lambda,f}^{\e_m,\w}(x)\right|^2}{|z|^{d+\alpha}}dzdx\\
&\le c_{15}\int_{\R^d}\int_{\R^d}\frac{|u_{\lambda,f}^{\e_m,\w}(x+z)-u_{\lambda,f}^{\e_m,\w}(x)|^2}{|z|^{d+\alpha}}
\left|\frac{\partial g(x+z)}{\partial x_j}\right|^2dzdx
+c_{15}\int_{\R^d}|u_{\lambda,f}^{\e_m,\w}(x)|^2\left(\int_{\R^d}\frac{\left|\frac{\partial g(x+z)}{\partial x_j}-\frac{\partial g(x)}{\partial x_j}\right|^2}{|z|^{d+\alpha}}dz\right)dx\\
&\le c_{16}\int_{\R^d}\int_{\R^d}\frac{|u_{\lambda,f}^{\e_m,\w}(x+z)-u_{\lambda,f}^{\e_m,\w}(x)|^2}{|z|^{d+\alpha}}
dzdx+
%c_{15}\int_{\R^d}\frac{|u_{\lambda,f}^{\e_m,\w}(x)|^2}{(1+|x|)^{d+\alpha}}dx.
c_{15}\int_{\R^d}|u_{\lambda,f}^{\e_m,\w}(x)|^2dx.
\end{align*}
Hence according to \eqref{t1-1-1a} we obtain
\begin{equation}\label{t1-2-9}
\sup_{m\ge 1}\int_{\R^d}|\xi|^\alpha \left|\F\left(\frac{\partial g(\cdot)}{\partial x_j}
u_{\lambda,f}^{\e_m,\w}(\cdot)\right)(\xi)\right|^2d\xi<+\infty.
\end{equation}

Meanwhile it holds that
\begin{align*}
&\quad\int_{\R^d}|\xi|^{2-\alpha}\left|\F\left(H_{jl}^{\e_m}(\cdot;\w)h(\cdot)\right)(\xi)\right|^2d\xi\\
&=\int_{\R^d}\int_{\R^d}\frac{\left|H_{jl}\left(\frac{x}{\e_m};\w\right)h(x)-H_{jl}\left(\frac{y}{\e_m};\w\right)h(y)\right|^2}{|x-y|^{d+2-\alpha}}dxdy\\
&\le 2\int_{\R^d}\int_{\R^d}\frac{\left|H_{jl}\left(\frac{x}{\e_m};\w\right)-H_{jl}\left(\frac{y}{\e_m};\w\right)\right|^2|h(y)|^2}{|x-y|^{d+2-\alpha}}dxdy
+2\int_{\R^d}\int_{\R^d}\frac{|h(x)-h(y)|^2}{|x-y|^{d+2-\alpha}}\left|H_{jl}\left(\frac{x}{\e_m};\w\right)\right|^2dxdy\\
&\le c_{17}\int_{\R^d}\int_{\R^d}\frac{\left|H_{jl}\left(\frac{x}{\e_m};\w\right)-H_{jl}\left(\frac{y}{\e_m};\w\right)\right|^2|h(y)|^2}{|x-y|^{d+2-\alpha}}dxdy
+c_{17}\int_{\R^d}\frac{\left|H_{jl}\left(\frac{x}{\e_m};\w\right)\right|^2}{(1+|x|)^{d+2-\alpha}}dx,
\end{align*}
where the last step we have used the property(since $h\in C_c^\infty(\R^d)$)
\begin{align*}
\quad \int_{\R^d}\frac{\left|h(x)-h(y)\right|^2}{|x-y|^{d+2-\alpha}}dy
&\le \|\nabla h\|_\infty^2\cdot \int_{\{y\in \R^d; |y-x|\le 1\}}\frac{|y-x|^2}{|y-x|^{d+2-\alpha}}dz\\
&+4\|h\|_\infty^2 \cdot \int_{\{y\in \R^d;|y-x|>1\}}\frac{1}{|x-y|^{d+2-\alpha}}dz\le c_{16},\ \forall x\in B(0,3R_0),
\end{align*}
\begin{align*}
\int_{\R^d}\frac{\left|h(x)-h(y)\right|^2}{|x-y|^{d+2-\alpha}}dx
&\le \|h\|_\infty^2\cdot\left(\int_{\{y\in \R^d; |y|\le 2R_0\}}\frac{1}{|x-y|^{d+2-\alpha}}dy\right)\le c_{17}(1+|x|)^{-d-2+\alpha},\ \forall\ x\notin B(0,3R_0).
\end{align*}
By change of variable we obtain
\begin{align*}
&\quad \int_{\R^d}\int_{\R^d}\frac{\left|H_{jl}\left(\frac{x}{\e_m};\w\right)-H_{jl}\left(\frac{y}{\e_m};\w\right)\right|^2|h(y)|^2}{|x-y|^{d+2-\alpha}}dxdy\\
&\le
\e_m^{d-(2-\alpha)}\|h\|_\infty^2\int_{B\left(0,\frac{2R_0}{\e_m}\right)}\left(\int_{\R^d}\frac{\left|H_{jl}\left(x;\w\right)-H_{jl}\left(y;\w\right)\right|^2}{|x-y|^{d+2-\alpha}}dy\right)dx\\
&=\e_m^{-(2-\alpha)}\e_m^d\|h\|_\infty^2\int_{B\left(0,\frac{2R_0}{\e_m}\right)}
\left(\int_{\R^d}\frac{\left|\tilde H_{jl}(\tau_{x+z}\w)-\tilde H_{jl}(\tau_{x}\w)\right|^2}{|z|^{d+2-\alpha}}dz\right)dx.
\end{align*}
According to ergodic theorem we derive
\begin{equation}\label{t2-1-8}
\begin{split}
&\quad \lim_{m \to \infty}\left|B\left(0,\frac{2R_0}{\e_m}\right)\right|^{-1}\int_{B\left(0,\frac{2R_0}{\e_m}\right)}
\left(\int_{\R^d}\frac{\left|\tilde H_{jl}(\tau_{x+z}\w)-\tilde H_{jl}(\tau_{x}\w)\right|^2}{|z|^{d+2-\alpha}}dz\right)dx\\
&=\Ee\left[\int_{\R^d}\frac{\left|\tilde H_{jl}(\tau_z \w)-\tilde H_{jl}(\w)\right|^2}{|z|^{d+2-\alpha}}dz\right]
\le \|\tilde H_{jl}\|_{1-\frac{\alpha}{2}}^2,
\end{split}
\end{equation}
which implies immediately that
\begin{align*}
\int_{\R^d}\int_{\R^d}\frac{\left|H_{jl}\left(\frac{x}{\e_m};\w\right)-H_{jl}\left(\frac{y}{\e_m};\w\right)\right|^2|h(y)|^2}{|x-y|^{d+2-\alpha}}dxdy
%\le c_{17}(\w)\e_m^{-(2-\alpha)}\|\tilde H_{jl}\|_{1-\frac{\alpha}{2}}^2\le
\le c_{18}(\w)\e_m^{-(2-\alpha)},
\end{align*}
where we also use the fact $\|\tilde H_{jl}\|_{1-\frac{\alpha}{2}}^2\le\|\tilde H_{jl}\|_{1}^2<\infty$ and $c_{18}(\w)$
is a positive constant which may depend on $\w$.

Using ergodic theorem we have
\begin{align*}
\varlimsup_{m \to \infty}\int_{B(0,2R_0)}\frac{\left|H_{jl}\left(\frac{x}{\e_m};\w\right)\right|^2}{(1+|x|)^{d+2-\alpha}}dx
&\le c_{19}\varlimsup_{m \to \infty}\int_{B(0,2R_0)}\left|\tilde H_{jl}\left(\tau_{\frac{x}{\e_m}}\w\right)\right|^2dx
=c_{19}|B(0,2R_0)|\cdot\Ee[|\tilde H_{jl}|^2].
\end{align*}
And by direct computation we obtain
\begin{align*}
&\quad \int_{B(0,2R_0)^c}\frac{\left|H_{jl}\left(\frac{x}{\e_m};\w\right)\right|^2}{(1+|x|)^{d+2-\alpha}}dx
=\frac{1}{|B(0,1)|}\int_{B(0,2R_0)^c}\left(\int_{B(0,1)}\frac{\left|H_{jl}\left(\frac{x}{\e_m};\w\right)\right|^2}{(1+|x|)^{d+2-\alpha}}dy\right)dx\\
&\le c_{20}\int_{B(0,2R_0)^c}\int_{B(0,1)}\frac{\left|H_{jl}\left(\frac{x}{\e_m};\w\right)-H_{jl}\left(\frac{y}{\e_m};\w\right)\right|^2}{|x-y|^{d+2-\alpha}}dydx\\
&\quad +c_{20}\int_{B(0,1)}\left|H_{jl}\left(\frac{y}{\e_m};\w\right)\right|^2
\left(\int_{B(0,2R_0)^c}\frac{1}{(1+|x|)^{d+2-\alpha}}dx\right)dy\\
&\le c_{21}\e_m^{-(2-\alpha)}\e_m^d\int_{B\left(0,\frac{1}{\e_m}\right)}
\left(\int_{\R^d}\frac{\left|\tilde H_{jl}(\tau_{x}\w)-\tilde H_{jl}(\tau_{y}\w)\right|^2}{|x-y|^{d+2-\alpha}}dx\right)dy
+c_{21}\int_{B(0,1)}\left|\tilde H_{jl}\left(\tau_{\frac{y}{\e_m}}\w\right)\right|^2dy,
\end{align*}
where in the last step we have used the change of variable.
Therefore based on this estimate and applying the same arguments of \eqref{t2-1-8} we obtain
\begin{align*}
\int_{B(0,2R_0)^c}\frac{\left|H_{jl}\left(\frac{x}{\e_m};\w\right)\right|^2}{(1+|x|)^{d+2-\alpha}}dx
\le c_{22}(\w)\e_m^{-(2-\alpha)}.
\end{align*}
Hence combining all above estimates together yields that for every $m\ge 1$,
\begin{align*}
\int_{\R^d}|\xi|^{2-\alpha}\left|\F\left(H_{jl}^{\e_m}(\cdot;\w)h(\cdot)\right)(\xi)\right|^2d\xi
\le c_{23}(\w)\e_m^{-(2-\alpha)}.
\end{align*}
According to  this and \eqref{t1-2-9} together we derive
\begin{align*}
\e_m^{2-\alpha}\left|\int_{\R^d}\frac{\partial}{\partial x_l}\left( H_{jl}^{\e_m}\left(\cdot;\w\right)h(\cdot)\right)(x)
\frac{\partial g(x)}{\partial x_j}u_{\lambda,f}^{\e_m,\w}(x)dx\right|\le c_{24}(\w)\e_m^{1-\alpha/2}.
\end{align*}
This, along with \eqref{t1-2-7a} yields \eqref{t1-2-7} and we have finished the proof.

\end{proof}

%\noindent {\bf Acknowledgements.}
%The research of Xin Chen is supported by the National Natural Science Foundation of China (No.\ 11871338).\

\vskip 0.3truein
{\small
{\bf Xin Chen:}
   School of Mathematical Sciences, Shanghai Jiao Tong University, 200240 Shanghai, P.R. China.
   \newline Email: \texttt{chenxin217@sjtu.edu.cn}

\bigskip

{\bf Kun Yin:}
    School of Mathematical Sciences, Shanghai Jiao Tong University, 200240 Shanghai, P.R. China.
    \newline Email: \texttt{epsilonyk@sjtu.edu.cn}

\end{document}